\documentclass{amsart}
\usepackage{amsmath,amssymb,amsfonts}
\usepackage[mathscr]{eucal}

\usepackage{enumerate}

\theoremstyle{plain}
\newtheorem{theorem}{Theorem}[section]

\newtheorem{prop}[theorem]{Proposition}

\theoremstyle{definition}
\newtheorem{definition}[theorem]{Definition}
\newtheorem{example}[theorem]{Example}

\numberwithin{equation}{section}

\input xy
\xyoption{all}

\newcommand{\Deltaop}{{\bf \Delta}^{op}}

\newcommand{\nerve}{\text{nerve}}

\newcommand{\we}{\text{we}}

\newcommand{\Map}{\text{Map}}
\newcommand{\Aut}{\text{Aut}}
\newcommand{\Ho}{\text{Ho}}
\newcommand{\SSets}{\mathcal{SS}ets}

\newcommand{\Sets}{\mathcal Sets}

\newcommand{\map}{\text{map}}
\newcommand{\ob}{\text{ob}}
\newcommand{\hoequiv}{\text{hoequiv}}

\newcommand{\css}{\mathcal{CSS}}

\begin{document}

\title[Homotopy fiber products]{Homotopy fiber products of homotopy theories}

\author[J.E. Bergner]{Julia E. Bergner}

\address{Department of Mathematics, University of California, Riverside, CA 92521}

\email{bergnerj@member.ams.org}

\date{\today}



\thanks{The author was partially supported by NSF grant DMS-0805951.}

\begin{abstract}
Given an appropriate diagram of left Quillen functors between model categories, one can define a notion of homotopy fiber product, but one might ask if it is really the correct one.  Here, we show that this homotopy pullback is well-behaved with respect to translating it into the setting of more general homotopy theories, given by complete Segal spaces, where we have well-defined homotopy pullbacks.
\end{abstract}

\maketitle

\section{Introduction}

Homotopy theory originates with the study of topological spaces up
to weak homotopy equivalence, where two spaces are weakly homotopy
equivalent if there is a map between them inducing isomorphisms on
all the homotopy groups.  The question of whether the nice properties of topological spaces held in
other settings led to the development of the notion of a model
category by Quillen \cite{quillen}.   The study of model
categories is then a more abstract form of homotopy theory, one
which has been investigated extensively.

One could then consider model categories themselves as objects of a category and consider Quillen equivalences as weak equivalences between them.  In this framework, one could ask questions about relationships between model categories; for example, what would a homotopy limit or homotopy colimit of a diagram of model categories be?  Unfortunately, there are no immediate answers to these questions because at present there is no known model structure on the category of model categories.  In one particular case of a homotopy limit, we have a plausible construction, that of the homotopy fiber product.  This construction was first explained to the author by Smith; it appears in the literature in a paper of To\"en \cite{toendha}, where it is useful in his development of derived Hall algebras.  One might ask whether, without a model category of model categories, we can really accept this definition as a genuine homotopy pullback construction.

However, one can also consider a homotopy theory to be something
more general than a model category, such as a category with a specified class of weak
equivalences, or maps one would like to consider as equivalences
but which are not necessarily isomorphisms.  While there may or
may not be a model structure on such a category, one can
heuristically think of formally inverting the weak equivalences, set-theoretic problems notwithstanding.  Furthermore, such ``homotopy theories" can be regarded as the objects of a model category.

In a series of papers \cite{dkcalc}, \cite{dkfncxes},
\cite{dksimploc}, Dwyer and Kan develop the theory of simplicial
localizations, which are simplicial categories corresponding to
model categories, or, more generally, categories with specified
weak equivalences. Since every such category has a simplicial
localization, and since, up to a natural notion of equivalence of
simplicial categories, every simplicial category arises as the
simplicial localization of some category with weak equivalences
\cite[2.5]{dkdiag}, simplicial categories can be considered to be
models for homotopy theories.  With these weak equivalences, often called Dwyer-Kan
equivalences, the category of (small) simplicial categories can
itself be considered as a homotopy theory, and in fact it has a model structure \cite{simpcat}.  Thus, within this framework one can address questions
about homotopy theories that are natural to ask in a given model
category, such as characterizations of homotopy limits and colimits.

Unfortunately, the model structure on the category of simplicial
categories is technically difficult to work with.  The weak
equivalences are particularly challenging to identify, for
example.  Fortunately, we have a choice of several other
equivalent model categories in which to address these questions.
Work of the author and of Joyal and Tierney has proved that the
complete Segal space model structure on the category of simplicial
spaces, two different Segal category model structures on the
category of Segal precategories, and the quasi-category model
structure on the category of simplicial sets are all equivalent as
model categories, and thus each is a model for the homotopy theory of
homotopy theories \cite{thesis}, \cite{joyalsc}, \cite{jt}.

In this paper, we address the construction of the homotopy fiber product of model categories and its analogue within the complete Segal space model structure.  Of the various models mentioned
above, the complete Segal space model structure is the best
setting in which to answer this question due to the particularly
nice description of the relevant weak equivalences.

Some clarification of terminology is needed here, as there are actually two notions of what is meant by a homotopy fiber product of model categories.  One is more restrictive than the other; the stronger is the one we would expect to correspond to a homotopy pullback, but the weaker is the one which can be given the structure of a model category.  We show that at least in some cases this model structure can be localized so that the fibrant-cofibrant objects satisfy the more restrictive condition.

The main result of this paper is that the stricter definition does in fact correspond to a homotopy pullback when we work in the complete Segal space setting.  We also give a characterization of the complete Segal spaces arising from the less restrictive description.

We should point out that this construction has been used in special cases, for example by H\"uttemann, Klein, Vogell, Waldhausen, and Williams in \cite{hkvww}, and as an example of more general constructions, for example the twisted diagrams of H\"uttemann and R\"ondigs \cite{hr}, and the model categories of comma categories as given by Stanculescu \cite{stan}.

It is also important to note that translating this question into the setting of more general homotopy theories is not merely a temporary solution until one finds a model category of model categories.  In practice, model structures are often hard to establish, and furthermore, the condition of having a Quillen pair between two such model structures is a rigid one.  Being able to consider homotopy theories as more flexible kinds of objects, and having morphisms between them less structured, makes it more likely that we can actually implement such a construction.  We consider such a case in Example \ref{fiber}.  Yet, with the relationship between the two settings established, we can use the additional structure when we do indeed have it.

In fact, part of our motivation for making the comparison in this paper is to generalize To\"en's development of derived Hall algebras.  Where he defines an associative algebra corresponding to stable model categories given by modules over a dg category, we would like to define such an algebra using a more general stable homotopy theory, namely, one given by a stable complete Segal space.  The main result of this paper allows us to use a homotopy pullback of complete Segal spaces in the setting where To\"en uses a homotopy fiber product of model categories.

The dual notion of homotopy pushouts of model categories, as well as more general homotopy limits and homotopy colimits, will be considered in later work.


\section{Model categories and more general homotopy theories}

In this section we give a brief review of model categories and their relationship with the complete Segal space model for more general homotopy theories.

Recall that a \emph{model category} $\mathcal M$ is a category with three distinguished classes of morphisms: weak equivalences, fibrations, and cofibrations, satisfying five axioms \cite[3.3]{ds}.  Given a model category structure, one can pass to the homotopy category $\Ho(\mathcal M)$, in which the weak equivalences from $\mathcal M$ become isomorphisms.  In particular, the weak equivalences, as the morphisms that we wish to invert, make up the most important part of a model category.  An object $x$ in $\mathcal M$ is \emph{fibrant} if the unique map $x \rightarrow \ast$ to the terminal object is a fibration.  Dually, an object $x$ in $\mathcal M$ is \emph{cofibrant} if the unique map $\phi \rightarrow x$ from the initial object is a cofibration.


Given a model category $\mathcal M$, there is also a model structure on the category $\mathcal M^{[1]}$, often called the morphism category of $\mathcal M$.  The objects of $\mathcal M^{[1]}$ are morphisms of $\mathcal M$, and the morphisms of $\mathcal M^{[1]}$ are given by pairs of morphisms making the appropriate square diagram commute.  A morphism in $\mathcal M^{[1]}$ is a weak equivalence (or cofibration) if its component maps are weak equivalences (or cofibrations) in $\mathcal M$.  More generally, $\mathcal M^{[n]}$ is the category with objects strings of $n$ composable morphisms in $\mathcal M$; the model structure can be defined analogously.

One could, more generally, consider categories with weak equivalences and no additional structure, and then formally invert the weak equivalences.  This process does give a homotopy category, but it frequently has the disadvantage of having a proper class of morphisms between any two given objects.  If we are willing to accept such set-theoretic problems, then we can work in this situation; the advantage of a model structure is that it provides enough additional structure so that we can take homotopy classes of maps and hence avoid these difficulties.  In this paper, we will use both notions of a ``homotopy theory," depending on the circumstances.

We would, however, like to work with nice objects modeling categories with weak equivalences.  While there are several options, the model that we use in this paper is that of complete Segal spaces.

Recall that the simplicial category $\Deltaop$ is defined to be the category with objects finite ordered sets $[n]=\{0 \rightarrow 1 \rightarrow \cdots \rightarrow n\}$ and morphisms the opposites of the order-preserving maps between them.  A \emph{simplicial set} is then a functor
\[ K \colon \Deltaop \rightarrow \Sets. \]
We denote by $\SSets$ the category of simplicial sets, and this category has a natural model category structure equivalent to the standard model structure on topological spaces \cite[I.10]{gj}.

One can just as easily consider more general simplicial objects; in this paper we consider \emph{simplicial spaces} (also called bisimplicial sets), or functors
\[ X \colon \Deltaop \rightarrow \SSets. \]
There are several model category structures on the category of bisimplicial sets.  An important one is the Reedy model structure \cite{reedy}, which is equivalent to the injective model structure, where the weak equivalences are given by levelwise weak equivalences of simplicial sets, and the cofibrations are given likewise \cite[15.8.7]{hirsch}.  Given a simplicial set $K$, we also denote by $K$ the simplicial space which has the simplicial set $K$ at every level.  We denote by $K^t$, or ``$K$-transposed", the constant simplicial space in the other direction, where $(K^t)_n = K_n$, where on the right-hand side $K_n$ is regarded as a discrete simplicial set.

We use the idea, originating with Dwyer and Kan, that a simplicial category, or category enriched over simplicial sets, models a homotopy theory, in the following way.  Using either of their two notions of simplicial localization, one can obtain from a category with weak equivalences a simplicial category \cite{dkfncxes}, \cite{dksimploc}, and there is a model structure $\mathcal{SC}$ on the category of all small simplicial categories \cite{simpcat}; thus, we have a so-called homotopy theory of homotopy theories.  One useful consequence of taking the simplicial category corresponding to a model category is that we can use it to describe \emph{homotopy function complexes}, or homotopy-invariant mapping spaces $\Map^h(x,y)$ between objects of a model category which is not necessarily equipped with the additional structure of a simplicial model category.  We use, in particular, the simplicial set $\Aut^h(x)$ of homotopy invertible self-maps of an object $x$.

Here we also consider simplicial spaces satisfying conditions imposing a notion of composition up to homotopy.  These Segal spaces and complete Segal spaces were first introduced by Rezk \cite{rezk}, and the name is meant to be suggestive of similar ideas first presented by Segal \cite{segal}.

\begin{definition} \cite[4.1]{rezk}
A \emph{Segal space} is a Reedy fibrant simplicial space $W$ such that the Segal maps
\[ \varphi_n \colon W_n \rightarrow \underbrace{W_1 \times_{W_0} \cdots \times_{W_0} W_1}_n \] are weak equivalences of simplicial sets for all $n \geq 2$.
\end{definition}

Given a Segal space $W$, we can consider its ``objects" $\ob(W)= W_{0,0}$, and, between any two objects $x$ and $y$, the ``mapping space" $\map_W(x,y)$, given by the homotopy fiber of the map $W_1 \rightarrow W_0 \times W_0$ given by the two face maps $W_1 \rightarrow W_0$.  The Segal condition given here tells us that a Segal space has a notion of $n$-fold composition of mapping spaces, at least up to homotopy.  More precise constructions are given by Rezk \cite[\S 5]{rezk}.  Using this composition, we can define ``homotopy equivalences" in a natural way, and then speak of the subspace of $W_1$ whose components contain homotopy equivalences, denoted $W_{\hoequiv}$.  Notice that the degeneracy map $s_0 \colon W_0 \rightarrow W_1$ factors through $W_{\hoequiv}$; hence we may make the following definition.

\begin{definition} \cite[\S 6]{rezk}
A \emph{complete Segal space} is a Segal space $W$ such that the map $W_0 \rightarrow W_{\hoequiv}$ is a weak equivalence of simplicial sets.
\end{definition}

Given this definition, we can describe the model structure on the category of simplicial spaces that is used throughout this paper.

\begin{theorem} \cite[\S 7]{rezk}
There is a model category structure $\css$ on the category of simplicial spaces, obtained as a localization of the Reedy model structure such that:
\begin{enumerate}
\item the fibrant objects are the complete Segal spaces,

\item all objects are cofibrant, and

\item the weak equivalences between complete Segal spaces are levelwise weak equivalences of simplicial sets.
\end{enumerate}
\end{theorem}

Now we return to the idea that a complete Segal space models a homotopy theory.

\begin{theorem} \cite{thesis}
The model categories $\mathcal{SC}$ and $\css$ are Quillen equivalent.
\end{theorem}

Furthermore, due to work of Rezk \cite{rezk} which was continued by the author \cite{css}, we can actually characterize, up to weak equivalence, the complete Segal space arising from a simplicial category, or more specifically, from a model category.  Rezk defines a functor which we denote $L_C$ from the category of model categories to the category of simplicial spaces; given a model category $\mathcal M$, we have that
\[ L_C(\mathcal M)_n = \nerve(\we(\mathcal M^{[n]})). \]  Here, $\mathcal M^{[n]}$ is defined as above, and $\we(\mathcal M^{[n]})$ denotes the subcategory of $\mathcal M^{[n]}$ whose morphisms are the weak equivalences.  While the resulting simplicial space is not in general Reedy fibrant, and hence not a complete Segal space, Rezk proves that taking a Reedy fibrant replacement is sufficient to obtain a complete Segal space \cite[8.3]{rezk}.  Hence, for the rest of this paper we assume that the functor $L_C$ includes composition with this Reedy fibrant replacement and therefore gives a complete Segal space.

One other difficulty that arises in this definition is the fact that it is only a well-defined functor on the category whose objects are model categories and whose morphisms preserve weak equivalences.  This restriction on the morphisms is stronger than we would like; it would be preferable to have such a functor defined on the category of model categories with morphisms Quillen functors, where weak equivalences are only preserved between either fibrant or cofibrant objects.  To obtain this more suitable functor, we consider $\mathcal M^c$, the full subcategory of $\mathcal M$ whose objects are cofibrant.  While $\mathcal M^c$ may no longer have the structure of a model category, it is still a category with weak equivalences, and therefore Rezk's definition can still be applied to it, so our new definition is
\[ L_C(\mathcal M)_n = \nerve(\we((\mathcal M^c)^{[n]})). \]  Each space in this new diagram is weakly equivalent to the one given by the previous definition, and now the construction is functorial on the category of model categories with morphisms the left Quillen functors.  

If one wanted to consider right Quillen functors instead, we could take the full subcategory of fibrant objects, $\mathcal M^f$, rather than $\mathcal M^c$.

Before stating the theorem giving the characterization, we give some facts about simplicial monoids, or functors
from $\Deltaop$ to the category of monoids.  Given a simplicial
monoid $M$
%
(or, more commonly, a simplicial group), we can find a classifying complex
of $M$, a simplicial set whose geometric realization is the
classifying space $BM$. A precise construction can be made for
this classifying space by the $\overline W M$ construction
\cite[V.4.4]{gj}, \cite{may}.  As we are not so concerned
here with the precise construction as with the fact that such a
classifying space exists, we will simply write
$BM$ for the classifying complex of $M$.

\begin{theorem} \cite[7.3]{css}  \label{baut}
Let $\mathcal M$ be a model category.  For $x$ an object of $\mathcal M$ denote by $\langle x \rangle$ the weak equivalence class of $x$ in $\mathcal M$, and denote by $Aut^h(x)$ the simplicial monoid of self weak equivalences of $x$.  Up to weak equivalence in the model category $\css$, the complete Segal space $L_C(\mathcal M)$ looks like
\[ \coprod_{\langle x \rangle} BAut^h(x) \Leftarrow \coprod_{\langle \alpha \colon x \rightarrow y \rangle} BAut^h(\alpha) \Lleftarrow \cdots. \]
\end{theorem}

(We should point out that the reference (Theorem 7.3 of \cite{css}) gives a characterization of the complete Segal space arising from a simplicial category, not from a model category.  However, the results of \S 6 of that same paper allow one to translate it to the theorem as stated here.)


This characterization, together with the fact that weak equivalences between complete Segal spaces are levelwise weak equivalences of simplicial sets, enables us to compare complete Segal spaces arising from different model categories.

\section{Homotopy fiber products of model categories}

We begin with the definition of homotopy fiber product as given by To\"en in \cite{toendha}.
First, suppose that
\[ \xymatrix@1{\mathcal M_1 \ar[r]^{F_1} & \mathcal M_3 & \mathcal
M_2 \ar[l]_{F_2}} \] is a diagram of left Quillen functors of
model categories.  Define their \emph{homotopy fiber product} to be the
model category $\mathcal M = \mathcal M_1 \times^h_{\mathcal M_3}
\mathcal M_2$ whose objects are given by 5-tuples $(x_1, x_2, x_3;
u, v)$ such that each $x_i$ is an object of $\mathcal M_i$ fitting
into a diagram
\[ \xymatrix@1{F_1(x_1) \ar[r]^-u & x_3 & F_2(x_2) \ar[l]_v.} \]
A morphism of $\mathcal M$, say $f \colon (x_1, x_2, x_3; u, v)
\rightarrow (y_1, y_2, y_3; z, w)$, is given by a triple of maps $f_i: x_i
\rightarrow y_i$ for $i=1,2,3$, such that the following diagram commutes:
\[ \xymatrix{F_1(x_1) \ar[r]^-u \ar[d]^{F_1(f_1)} & x_3 \ar[d]^{f_3} &
F_2(x_2) \ar[l]_-v \ar[d]^{F_2(f_2)} \\
F_1(y_1) \ar[r]^-z & y_3 & F_2(y_2) \ar[l]_-w.} \]

This category $\mathcal M$ can be given the structure of a model
category, where the weak equivalences and cofibrations are given
levelwise.  In other words, $f$ is a weak equivalence (or
cofibration) if each map $f_i$ is a weak equivalence (or
cofibration) in $\mathcal M_i$.

A more restricted definition of this construction requires that
the maps $u$ and $v$ be weak equivalences in $\mathcal M_3$.  Unfortunately, if we impose this additional condition, the resulting category cannot be given the structure of a model category because it is not closed under limits and colimits.  However, intuition suggests that we really want to require $u$ and $v$ to be weak equivalences in order to get an appropriate homotopy pullback.  We would like to have a localization of the model structure on $\mathcal M$ described above such that the fibrant-cofibrant objects have the maps $u$ and $v$ weak equivalences.  In at least some situations, we can find such a localization.

Recall that a model category is \emph{combinatorial} if it is cofibrantly generated and locally presentable as a category \cite[2.1]{duggercomb}.

\begin{theorem}
Let $\mathcal M$ be the homotopy fiber product of a diagram of left Quillen functors
\[ \xymatrix@1{\mathcal M_1 \ar[r]^{F_1} & \mathcal M_3 & \mathcal M_2 \ar[l]_{F_2} } \] where each of the categories $\mathcal M_i$ is combinatorial.  Further assume that $\mathcal M$ is right proper. Then there exists a right Bousfield localization of $\mathcal M$ whose fibrant and cofibrant objects $(x_1, x_2, y_2; u, v)$ have both $u$ and $v$ weak equivalences in $\mathcal M_3$.
\end{theorem}

\begin{proof}
Since the categories $\mathcal M_1$, $\mathcal M_2$, and $\mathcal M_3$ are combinatorial, and hence locally presentable, we can find, for each $i=1,2,3$, a set $\mathcal A_i$ of objects of $\mathcal M_i$ which generates all of $\mathcal M_i$ by filtered colimits \cite[2.2]{duggercomb}.  Furthermore, we can assume that the objects of $\mathcal A_i$ are all cofibrant.  (An explicit such set can be found, for example, using Dugger's notion of a presentation of a combinatorial model category \cite{dugger}.)  Given $a_1 \in \mathcal A_1$ and $a_2 \in \mathcal A_2$, consider the class of all objects $x_3$ such that there are pairs of weak equivalences
\[ \xymatrix@1{F_1(a_1) \ar[r]^\simeq  & x_3 & F_2(a_2). \ar[l]_\simeq} \]  Since $\mathcal A_1$ and $\mathcal A_2$ are sets, we can choose one representative of $x_3$ for each pair $a_1$ and $a_2$ with $F_1(a_1)$ weakly equivalent to $F_2(a_2)$.  Taking the union of this set together with the generating set $\mathcal A_3$ for $\mathcal M_3$, we obtain a set which we denote $\mathcal B_3$.  For $i=1,2$, let $\mathcal B_i=\mathcal A_i$.

In $\mathcal M$, consider the following set of objects:
\[ \{(x_1, x_2, x_3;u,v) \mid x_i \in \mathcal B_i, u, v \text{ weak equivalences in } \mathcal M_3\}. \]  By taking filtered colimits, we can obtain from this set all objects $(x_1, x_2, x_3;u,v)$ of $\mathcal M$ for which the maps $u$ and $v$ are weak equivalences; while arbitrary colimits do not necessarily preserve these weak equivalences, filtered colimits do \cite[7.3]{duggercomb}.
Thus, we can take a right Bousfield localization of $\mathcal M$ with respect to this set of objects; if $\mathcal M$ is right proper, then this localization has a model structure \cite[5.1.1]{hirsch}, \cite{barwick}.
\end{proof}

Unfortunately, it seems to be difficult to describe conditions on the model categories $\mathcal M_1$, $\mathcal M_2$, and $\mathcal M_3$ guaranteeing that $\mathcal M$ is right proper.  We can weaken this condition somewhat, using a remark of Hirschhorn \cite[5.1.2]{hirsch}.  Alternatively, Barwick discusses the structure which is retained after taking a right Bousfield localization of a model category which is not necessarily right proper \cite{barwick}.   Nonetheless, when the conditions of this theorem are not satisfied, we can still use the original levelwise model structure on $\mathcal M$ and simply restrict to the appropriate subcategory when we want to require $u$ and $v$ to be weak equivalences.

In order to determine whether this construction really
gives a homotopy fiber product of homotopy theories, we need to
translate it into the complete Segal space model structure via the functor $L_C$. When we require the
maps $u$ and $v$ to be weak equivalences, we can still take the associated complete Segal space even without a model structure, and we do get a
homotopy pullback in the model category $\css$.  The proof of this
statement is given in the next section.  However, the more general
construction also has a precise description as well, which we give
in the following section.

We conclude this section with a few examples.

\begin{example}
We begin with some comments on the use of homotopy fiber products of model categories as used by To\"en to prove associativity of his derived Hall algebras \cite{toendha}.  In this situation, we have a stable model category; this extra assumption that the homotopy category is triangulated implies that our model category has a ``zero object" so that the initial and terminal objects coincide.  We denote this object 0.

Let $T$ be a dg category, or category enriched over chain complexes over a finite field $k$.  Then a dg module over $T$ is a dg functor $T \rightarrow C(k)$, where $C(k)$ denotes the category of chain complexes of modules over $k$.  There is a model structure $\mathcal M(T)$ on the category of such modules over a fixed $T$, where the weak equivalences and fibrations are given levelwise \cite[\S 3]{toendg}.

Given an object of $\mathcal M(T)^{[1]}$, namely a map $f \colon x \rightarrow y$, let $F \colon \mathcal M(T)^{[1]} \rightarrow \mathcal M(T)$ be the target map, so that $F(f \colon x \rightarrow y)= y$.  Let $C \colon \mathcal M(T)^{[1]} \rightarrow \mathcal M(T)$ be the cone map, so that $C(f\colon x \rightarrow y)= y \amalg_x 0$.  Using these functors, we get a diagram
\[ \xymatrix{ & \mathcal M(T)^{[1]} \ar[d]^C \\
\mathcal M(T)^{[1]} \ar[r]^F & \mathcal M(T).} \]
To understand the homotopy fiber product $\mathcal M$ of this diagram, To\"en uses the model structure on the homotopy fiber product given by levelwise maps; eventually in the proof he adds the additional assumption that the maps $u$ and $v$ in the definition be weak equivalences \cite[\S 4]{toendha}.  The homotopy fiber product $\mathcal M$ given by this diagram is equivalent to the model category $\mathcal M(T)^{[2]}$ whose objects are pairs of composable morphisms in $\mathcal M(T)$.


\end{example}

\begin{example} \label{fiber}
Here we consider the following special case of a homotopy pullback, the homotopy fiber of a map.  Therefore, this definition of homotopy fiber product of model categories leads to the following definition.

\begin{definition}
Let $F \colon \mathcal M \rightarrow \mathcal N$ be a left Quillen functor of model categories.  Then the \emph{homotopy fiber} of $F$ is the homotopy fiber product of the diagram
\[ \xymatrix{ & \mathcal M \ar[d]^F \\
\ast \ar[r] & \mathcal N} \] where the map $\ast \rightarrow N$ is necessarily the map from the trivial model category to the initial object $\phi$ of $\mathcal N$.
\end{definition}

Using our definition, the objects of this homotopy fiber are triples $(\ast, m, n; u, v)$, where $\ast$ denotes the single object of the trivial model category $\ast$, $m$ is an object of $\mathcal M$, $n$ is an object of $\mathcal N$, $u \colon \phi \rightarrow n$ is the unique such map, and $v \colon F(m) \rightarrow n$.  Imposing our condition that $u$ and $v$ be weak equivalences, we get that $n$ must be weakly equivalent to the initial object of $\mathcal N$, and $m$ is any object of $\mathcal M$ whose image under $F$ is weakly equivalent to the initial object of $\mathcal N$.

While this definition follows naturally from the usual notions, it is unsatisfactory for many purposes.  The requirement that the functors in the pullback diagram be left Quillen is a very rigid one. One might perhaps prefer to look at the homotopy fiber over some other object, but here one cannot.
\end{example}

\begin{example}
A further specialization of this definition illustrates its particularly odd nature.  If we take the analogue of a loop space and define the ``loop model category" as the homotopy pullback of the diagram
\[ \xymatrix{ & \ast \ar[d] \\
\ast \ar[r] & \mathcal M} \] for any model category $\mathcal M$, we simply get the subcategory of $\mathcal M$ whose objects are weakly equivalent to the initial object.
\end{example}

\section{Homotopy pullbacks of complete Segal spaces}

Consider the functor $L_C$ which takes a model category (or
simplicial category) to a complete Segal space.  Given a homotopy
fiber square of model categories as defined in the previous
section (namely, where we require the maps $u$ and $v$ to be weak equivalences), we can apply this functor to obtain a homotopy commutative square
\[ \xymatrix{L_C \mathcal M \ar[r] \ar[d] & L_C \mathcal M_2
\ar[d] \\
L_C \mathcal M_1 \ar[r] & L_C \mathcal M_3.} \]

Alternatively, we could apply the functor $L_C$ only to the
original diagram and take the homotopy pullback, which we denote
$P$, and obtain the following diagram:
\[ \xymatrix{P \ar[r] \ar[d] & L_C \mathcal M_2
\ar[d] \\
L_C \mathcal M_1 \ar[r] & L_C \mathcal M_3.} \]

Since $P$ is a homotopy pullback, there exists a natural map $L_C \mathcal M \rightarrow P$.

\begin{theorem} \label{main}
The map
\[ L_C \mathcal M \rightarrow P = L_C \mathcal M_1 \times^h_{L_C \mathcal M_3} L_C \mathcal M_2 \] is a weak equivalence of complete Segal spaces.
\end{theorem}

To prove this theorem, we would like to be able to use Theorem \ref{baut} which characterizes the complete Segal spaces that result from applying the functor $L_C$ to a model category.  However, this theorem only gives the homotopy type of each space in the simplicial diagram, not an explicit description of the precise spaces we obtain.  Thus, we begin by unpacking this characterization in order to obtain actual maps between these ``nice'' versions of the complete Segal spaces $L_C \mathcal M_i$.  More details can be found in \cite[\S 7]{css}, where Theorem \ref{baut} is proved.

We begin with the description of the space at level zero.  Given a model category $\mathcal M$, we can take its corresponding simplicial category $L\mathcal M$ given by Dwyer-Kan simplicial localization.  Denote by $\mathcal C$ the sub-simplicial category of $L\mathcal M$ whose morphisms are all invertible up to homotopy.  Then there is a weak equivalence $F(\mathcal C) \rightarrow \mathcal C$, where $F(\mathcal C)$ denotes the free simplicial category on $\mathcal C$ \cite[\S 2]{dksimploc}.  Then taking a groupoid completion of $F(\mathcal C)$ gives a simplicial groupoid $F(\mathcal C)^{-1} F(\mathcal C)$.  The characterization of the corresponding complete Segal space uses the fact that this simplicial groupoid is equivalent to one which is the disjoint union of simplicial groups, say $\mathcal G$.  There is a functor $F(\mathcal C)^{-1}F(\mathcal C) \rightarrow \mathcal G$ collapsing each component down to one with a single object.  Thus, we obtain a zig-zag of weak equivalences of simplicial categories
\[ \mathcal G \leftarrow F(\mathcal C)^{-1}F(\mathcal C) \leftarrow F(\mathcal C) \rightarrow \mathcal C. \]  Since all these constructions are functorial, by using Theorem \ref{baut}, we are essentially passing from working with $\mathcal C$ to working with $\mathcal G$.

We can apply these same constructions to the morphism category $\mathcal M^{[1]}$ to understand the space at level one, and, more generally, to $\mathcal M^{[n]}$ to obtain the description of the space at level $n$.

\begin{proof}[Proof of Theorem \ref{main}]
Since all the objects in question are complete Segal spaces, i.e., local objects in the model structure $\css$, it suffices to
show that the map $L_C \mathcal M \rightarrow P$ is a levelwise weak equivalence of simplicial sets.
Let us begin by comparing the space at level zero for each.  The
space $P_0$ looks like
\[ (L_C \mathcal M_1)_0 \times^h_{(L_C \mathcal M_3)_0} (L_C
\mathcal M_2)_0 = \coprod_{\langle x_1 \rangle} B\Aut^h(x_1)
\times^h_{\displaystyle{\coprod_{\langle x_3 \rangle} B \Aut^h(x_3)}}
\coprod_{\langle x_2 \rangle} B \Aut^h(x_2). \]  On the other
hand, $(L_C \mathcal M)_0$ looks like
\[ \coprod_{\langle (x_1, x_2, x_3; u,v) \rangle} B\Aut^h ((x_1,
x_2, x_3; u,v)). \]  However, since the classifying space functor $B$ commutes with
taking the disjoint union, this space is equivalent to
\[ B \left( \coprod _{\langle (x_1, x_2, x_3; u,v) \rangle} \Aut^h ((x_1,
x_2, x_3; u,v)) \right). \] Thus, $(L_C \mathcal M)_0$ looks like the
nerve of the category whose objects are diagrams of the form
\[ \xymatrix{F_1(x_1) \ar[r]^u \ar[d]^{F_1(a_1)} & x_3
\ar[d]^{a_3} & F_2(x_2) \ar[l]_v \ar[d]^{F_2(a_2)} \\
F_1(x_1) \ar[r]^u & x_3 & F_2(x_2) \ar[l]_v} \] where each $a_i
\in \Aut^h(x_i)$.  In other words,
\[ \Aut^h((x_1, x_2, x_3; u,v)) \] consists of triples $(a_1, a_2,
a_3)$ such that the above diagram commutes.

For the moment, let us suppose that we have no homotopy invariance problems and that $P_0$ can be given by a pullback, rather than a homotopy pullback; further explanation on this point will be given shortly.

Since $B$ is a right adjoint functor (see \cite[III.1]{gj} for details), it commutes with pullbacks, and we have that
\[ \begin{aligned}
P_0 & \simeq \coprod_{\langle x_1 \rangle} B\Aut^h(x_1)
\times_{\displaystyle{\coprod_{\langle x_3 \rangle} B\Aut^h(x_3)}}
\coprod_{\langle x_2 \rangle} B \Aut^h(x_2) \\
& \simeq B \left(\coprod_{\langle x_1 \rangle} \Aut^h(x_1)
\times_{\displaystyle{\coprod_{\langle x_3 \rangle} \Aut^h(x_3)}}
\coprod_{\langle x_2 \rangle} \Aut^h(x_2) \right) \\
& \simeq B \left(\coprod_{\langle x_1 \rangle, \langle x_2
\rangle, \langle x_3 \rangle} \Aut^h(x_1) \times_{\Aut^h(x_3)}
\Aut^h(x_2) \right).
\end{aligned} \]
Thus, $P_0$ also looks like the nerve of the category whose
objects are diagrams of the form given above, since the leftmost
and rightmost vertical arrows are indexed by maps in $\Aut(x_1)$
and $\Aut(x_2)$, not by their images in $\mathcal M_3$.  So, if
$F_1$, for example, identifies two maps of $\Aut(x_1)$, we
still count two different diagrams.  However, if we are taking a strict pullback, the horizontal maps must be equalities.  We claim that taking the homotopy pullback, rather than the strict pullback,  gives precisely all the diagrams as given above, without this restriction, as follows.

Since our diagram consists of fibrant objects in $\css$, we can apply \cite[19.9.4]{hirsch} and obtain a homotopy pullback by replacing one of the maps in the diagram with a fibration.  In doing so, an object $x_1$, for example, is replaced by a pair given by $x_1$ together with a map $F_1(x_1) \rightarrow x_3$.  Doing the same for the other map (since there is no harm in replacing both of them by fibrations) we obtain all diagrams of the form given above.  So, we have shown that we have the desired weak equivalence on level zero.

Now, it remains to show that we also get a weak equivalence of
spaces at level one.  The argument here is essentially the same
but with larger diagrams.  Again, we take ordinary pullbacks to reduce notation, but this issue can be resolved just as in the level zero case.

The space $P_1 = (L_C \mathcal M_1 \times_{L_C \mathcal M_3} L_C
\mathcal M_2)_1$ can be written as follows:

\SMALL

\begin{multline*} \left( \coprod_{\langle f_1 \colon x_1 \rightarrow y_1 \rangle} B \Aut^h (f_1) \right)
\times_{\displaystyle{\left( \coprod_{\langle f_3 \colon x_3 \rightarrow
y_3 \rangle} B \Aut^h (f_3)
\right)}} \left(
\coprod_{\langle f_2 \colon x_2 \rightarrow x_2 \rangle} B \Aut^h (f_2)\right) \\
\simeq B \left( \coprod_{\langle f_i \colon x_i \rightarrow y_i \rangle}
\left( \Aut^h (f_1) \right) \times_{\displaystyle{\Aut^h (f_3)}} \Aut^h (f_2) \right).
\end{multline*}

\normalsize

Note that when we take $\langle f_i \colon x_i \rightarrow y_i \rangle$, the notation is meant to signify that we are varying $x_i$ and $y_i$ as objects, as well as maps between them, and then taking distinct weak equivalence classes.

On the other hand, if we let
\[ f=(f_1, f_2, f_3) \colon (x_1, x_2, x_3; u,v) \rightarrow (y_1, y_2, y_3; w,z), \]
the space $(L_C \mathcal M)_1$ can be written as
\[ \coprod_{\langle f \rangle} B \Aut^h (f) \simeq
B \left( \coprod_{\langle f \rangle} \Aut^h (f) \right). \]

As above, let $a_i$ denote a homotopy
automorphism of $x_i$, and let $b_i$ denote a homotopy automorphism of $y_i$.  Then, both of the above spaces are given
by the nerve of the category whose objects are diagrams of the form
\[ \xymatrix{ F_1(x_1) \ar[rr]^u \ar[dd]_<<<<<<{F_1(f_1)}
\ar[dr]^{F_1(a_1)} && x_3 \ar[dr]^{a_3} \ar'[d]_{f_3}[dd] &&
F_2(x_2)
\ar[dr]^{F_2(a_2)} \ar'[d]_{F_2(f_2)}[dd] \ar[ll]_v & \\
& F_1(x_1) \ar[rr]^<<<<<<<<u \ar[dd]_<<<<<{F_1(f_1)} && x_3
\ar[dd]^<<<<<<{f_3} &&
F_2(x_2) \ar[ll]_<<<<<<<<<<v \ar[dd]^<<<<<<{F_2(f_2)} \\
F_1(y_1) \ar[dr]_{F_1(b_1)} \ar'[r][rr]^<<<<w && y_3 \ar[dr]^{b_3}
&&
F_2(y_2) \ar'[l]_>>>>>z[ll] \ar[dr]^{F_2(b_2)} & \\
& F_1(y_1) \ar[rr]^w && y_3 && F_2(y_2) \ar[ll]_z} \]

One could show that the higher-degree spaces of each of these
complete Segal spaces are also weakly equivalent, but since these
spaces are determined by these two, the above arguments are
sufficient.
\end{proof}

\section{The more general construction on complete Segal spaces}

In this section, we drop the condition that the maps $u$ and $v$
in the definition of the homotopy fiber product are weak
equivalences in $\mathcal M_3$ and give a characterization of the resulting complete Segal space.

Again, let
\[ \xymatrix{ & \mathcal M_2 \ar[d]^{F_2} \\
\mathcal M_1 \ar[r]^{F_1} & \mathcal M_3} \] be a diagram of model
categories and left Quillen functors.  Let $\mathcal
N$ be the category whose objects are given by 5-tuples $(x_1, x_2, x_3; u,
v)$, where $x_i$ is an object of $\mathcal M_i$ for each $i$, and
the maps $u$ and $v$ fit into a diagram
\[ \xymatrix@1{F_1(x_1) \ar[r]^u & x_3 & F_2(x_2) \ar[l]_v .} \]

The 0-space of the complete Segal space $L_C \mathcal N$ has the homotopy type
\[ \coprod_{\langle (x_1, x_2, x_3; u,v) \rangle} B \Aut^h((x_1,
x_2, x_3; u,v)). \]  An element of the group $\Aut^h((x_1, x_2,
x_3; u,v))$ looks like a diagram
\[ \xymatrix{F_1(x_1) \ar[r]^u \ar[d]_\simeq & x_3 \ar[d]^\simeq &
F_2(x_2) \ar[d]^\simeq \ar[l]_v \\
F_1(x_1) \ar[r]^u & x_3 & F_2(x_2) \ar[l]_v.} \]  Using this
diagram as a guide, we can formulate a concise
description of $(L_C \mathcal N)_0$.

\begin{prop} \label{level0}
Let $N_0$ denote the nerve of the category given by $(\cdot
\rightarrow \cdot \leftarrow \cdot)$.  The space $(L_C \mathcal
N)_0$ has the homotopy type of the pullback of the diagram
\[ \xymatrix{ & \Map(N_0^t, L_C \mathcal M_3) \ar[d] \\
\Map(\Delta [0]^t, L_C \mathcal M_1) \times \Map(\Delta [0]^t, L_C
\mathcal M_2) \ar[r] & \Map(\Delta [0]^t, L_C \mathcal M_3)^2} \]
where the horizontal map is given by
\[ \Map(\Delta[0]^t, L_C F_1) \times \Map(\Delta [0]^t, L_C F_2)
\] and the vertical arrow is induced the pair of source maps
$(s_1, s_2) \colon N_0 \rightarrow \Delta [0] \amalg \Delta [0]$.
\end{prop}

\begin{proof}
For $i=1,2,3$, let $a_i$ denote a homotopy automorphism of $x_i$ in $\mathcal M_i$.  The collection of diagrams of the form
\[ \xymatrix{F_1(x_1) \ar[r]^u \ar[d]^{F_1(a_1)} & x_3 \ar[d]^{a_3} &
F_2(x_2) \ar[d]^{F_2(a_2)} \ar[l]_v \\
F_1(x_1) \ar[r]^u & x_3 & F_2(x_2) \ar[l]_v.} \] can be written as
the pullback
\[ \Aut^h (u) \times_{\displaystyle{\Aut^h(x_3)}} \Aut^h (v). \]

Taking classifying spaces and coproducts over all isomorphism
classes of objects, we obtain the pullback

\setcounter{equation}{1}

\begin{equation} \label{pb}
\coprod_{\langle u \colon F_1(x_1) \rightarrow x_3 \rangle}  B \Aut^h
(u) \times_{\displaystyle{\coprod_{\langle x_3 \rangle} B
\Aut^h(x_3)}} \coprod_{\langle v \colon F_2(x_2) \rightarrow x_3 \rangle} B
\Aut^h (v).
\end{equation}

However, notice that the space
\[ \coprod_{\langle u \colon F_1(x_1) \rightarrow x_3 \rangle} B \Aut^h
(u) \] is equivalent to the pullback
\[ \coprod_{\langle x_1 \rangle} B \Aut^h (x_1)
\times_{\displaystyle{ \coprod_{\langle x_3 \rangle} B
\Aut^h(x_3)}} \coprod_{\langle f_3 \colon y_3 \rightarrow x_3 \rangle} B
\Aut^h (f_3), \]  since an element of $\Aut^h(u)$ looks like a diagram
\[ \xymatrix{ y_3 \ar[r]^{f_3} \ar[d]_\simeq & x_3 \ar[d]^\simeq \\
   y_3 \ar[r]^{f_3} & x_3} \] where $y_3 = F_1(x_1)$.

Analogously, the space
\[ \coprod_{\langle v \colon F_2(x_2) \rightarrow x_3 \rangle} B \Aut^h (v) \]
is equivalent to the pullback
\[ \coprod_{\langle x_2 \rangle} B \Aut^h (x_2)
\times_{\displaystyle{ \coprod_{\langle x_3 \rangle} B
\Aut^h(x_3)}} \coprod_{\langle f_3 \colon y_3 \rightarrow x_3 \rangle} B
\Aut^h (f_3). \]

Putting these two equivalences together, we get that the
pullback (\ref{pb}) can be written as

\Small

\begin{multline*}
\left(\coprod_{\langle x_1 \rangle} B \Aut^h (x_1)
\times_{\displaystyle{\coprod_{\langle x_3 \rangle} B
\Aut^h(x_3)}} \coprod_{\langle f_3 \rangle} B
\Aut^h (f_3) \right) \\ \times_{\displaystyle{\coprod_{\langle x_3
\rangle} B \Aut^h(x_3)}} \\ \left(\coprod_{\langle x_2 \rangle} B \Aut^h (x_2) \times_{\displaystyle{\coprod_{\langle x_3 \rangle} B \Aut^h(x_3)}} \coprod_{\langle f_3 \rangle} B\Aut^h(f_3) \right).
\end{multline*}

\normalsize

However, this pullback can be written in a much more manageable
way using our characterization of the complete Segal spaces
corresponding to a model category.  Thus, we get a pullback
\[ (L_C \mathcal M_1)_0 \times_{(L_C \mathcal M_3)_0} (L_C
\mathcal M_3^{[1]})_0 \times_{(L_C \mathcal M_3)_0} (L_C \mathcal M_2)_0
\times_{(L_C \mathcal M_2)_0} (L_C \mathcal M_3^{[1]})_0. \] Rearranging
terms in the pullback gives an equivalent formulation of this
space as
\[ ((L_C \mathcal M_1)_0 \times (L_C \mathcal M_2)_0) \times_{(L_C
\mathcal M_3)_0^2} ((L_C \mathcal M_3^{[1]})_0 \times_{(L_C \mathcal
M_3)_0} (L_C \mathcal M_3^{[1]})_0). \]  However, this space is
precisely the pullback of the diagram given in the statement of
the proposition, since $\Map(\Delta [0]^t, L_C \mathcal M_1)= (L_C
\mathcal M_1)_0$ and analogously for $\mathcal M_2$, and the
pullback on the right agrees with the space $\Map(N_0^t, L_C
\mathcal M_3)$.
\end{proof}

Now we give a characterization of the space $(L_C \mathcal N)_1$.

\setcounter{theorem}{2}

\begin{prop} \label{level1}
If $N_1$ denotes the nerve of the category given by
\[ \xymatrix{\cdot \ar[r] \ar[d] & \cdot \ar[d] & \cdot \ar[d]
\ar[l] \\
\cdot \ar[r] & \cdot & \cdot \ar[l]} \] then the space $(L_C \mathcal
N)_1$ is weakly equivalent to the homotopy pullback of the diagram
\[ \xymatrix{ & \Map(N_1^t, L_C \mathcal M_3) \ar[d] \\
\Map(\Delta [1]^t, L_C \mathcal M_1) \times \Map(\Delta [1]^t, L_C
\mathcal M_2) \ar[r] & \Map(\Delta [1]^t, L_C \mathcal M_3)^2} \]
where the maps are analogous to the ones in the previous
proposition.
\end{prop}

\begin{proof}
Again, let
\[ f = (f_1, f_2, f_3) \colon (x_1, x_2, x_3; u,v) \rightarrow (y_1, y_2, y_3; w, z). \]
Notice that, by definition, the homotopy type of the space
$(L_C \mathcal N)_1$ is given by
\[ \coprod_{\langle f \rangle} B \Aut^h (f). \]  An element of the group $\Aut^h (f)$ is given by a diagram
\[ \xymatrix{ F_1(x_1) \ar[rr]^u \ar[dd]_<<<<<<{F_1(f_1)}
\ar[dr]^{F_1(a_1)} && x_3 \ar[dr]^{a_3} \ar'[d]_{f_3}[dd] &&
F_2(x_2)
\ar[dr]^{F_2(a_2)} \ar'[d]_{F_2(f_2)}[dd] \ar[ll]_v & \\
& F_1(x_1) \ar[rr]^<<<<<<<<u \ar[dd]_<<<<<{F_1(f_1)} && x_3
\ar[dd]^<<<<<<{f_3} &&
F_2(x_2) \ar[ll]_<<<<<<<<<<v \ar[dd]^<<<<<<{F_2(f_2)} \\
F_1(y_1) \ar[dr]_{F_1(b_1)} \ar'[r][rr]^<<<<w && y_3 \ar[dr]^{b_3}
&&
F_2(y_2) \ar'[l]_>>>>>z[ll] \ar[dr]^{F_2(b_2)} & \\
& F_1(y_1) \ar[rr]^w && y_3 && F_2(y_2) \ar[ll]_z.} \]

If we let $\alpha_1 \colon u \rightarrow w$ and $\alpha_2 \colon v \rightarrow z$ be maps in $\mathcal M^{[1]}$, such a diagram can also be regarded as an element of the homotopy
pullback

\small

\[ \Aut^h(f_1) \times_{\Aut^h(f_3)} \Aut^h(\alpha_1)
\times_{\Aut^h(f_3)} \Aut^h(\alpha_2) \times_{\Aut^h(f_2)}
\Aut^h (f_3). \]

\normalsize

Taking classifying spaces and coproducts over all possible classes
of objects and morphisms, we obtain a pullback

\Small

\begin{multline*}
\left(\coprod_{\langle f_1 \rangle} B\Aut^h(f_1) \times_{\displaystyle{\coprod_{\langle f_3 \rangle}
B\Aut^h(f_3)}} \coprod_{\langle \alpha_1 \rangle} B\Aut^h(\alpha_1) \right) \\ \times_{\displaystyle{\coprod_{\langle f_3 \rangle}
B\Aut^h(f_3)}} \\ \left(\coprod_{\langle f_2 \rangle} B\Aut^h (f_2) \times_{\displaystyle{\coprod_{\langle f_3 \rangle}
B\Aut^h(f_3)}} \coprod_{\langle \alpha_2 \rangle} B\Aut^h(\alpha_2) \right).
\end{multline*}

\normalsize

This pullback can be rewritten in terms of the corresponding
complete Segal spaces as
\[ \left((L_C \mathcal M_3)_1 \times_{(L_C \mathcal M_3)_1} (L_C \mathcal
M_3^{[1]})_1 \right) \times_{(L_C \mathcal M_3)_1} \left( (L_C
\mathcal M_2)_1 \times_{(L_C \mathcal M_3)_1} (L_C \mathcal
M_3^{[1]})_1 \right). \]

At this point, notice that this space is also the pullback of
the diagram
\[ \xymatrix{ & (L_C \mathcal M_3^{[1]})_1 \times_{(L_C \mathcal
M_3)_1} (L_C \mathcal M_3^{[1]})_1 \ar[d] \\
(L_C \mathcal M_1)_1 \times_{(L_C \mathcal M_3)_1} (L_C \mathcal
M_2)_1 \ar[r] & (L_C \mathcal M_3)_1^2.} \]  However, since the
upper space is equivalent to $\Map(N_1^t, L_C \mathcal M_3)$, we have
completed the proof.
\end{proof}

Lastly, notice that the simplicial set $N_1$ from Proposition \ref{level1} is just $\Map(\Delta [1], N_0) = N_0^{\Delta[1]}$, where $N_0$ is as in Proposition \ref{level0}.  We can use these results and the properties of complete Segal spaces to give the following theorem.

\begin{theorem}
Let $N$ be the simplicial space given by $N_n = N_0^{\Delta [n]}$.  Then the complete Segal space $L_C \mathcal N$ is weakly equivalent to the homotopy pullback of the diagram
\[ \xymatrix{ & \Map(N, L_C \mathcal M_3) \ar[d] \\
L_C \mathcal M_1 \times L_C \mathcal M_2 \ar[r] & L_C\mathcal M_3 \times L_C \mathcal M_3.} \]
\end{theorem}

\end{document}